\documentclass{amsart}
\usepackage{amsmath,amsthm}
\usepackage{amsfonts,amssymb}
\usepackage{enumerate}
\usepackage{accents,color}
\usepackage{graphicx}

\newcommand{\lvt}{\left|\kern-1.35pt\left|\kern-1.3pt\left|}
\newcommand{\rvt}{\right|\kern-1.3pt\right|\kern-1.35pt\right|}
\newtheorem{thm}{Theorem}[section]
\newtheorem{cor}[thm]{Corollary}
\newtheorem{lem}[thm]{Lemma}
\newtheorem{prop}[thm]{Proposition}

\newtheorem{defn}[thm]{Definition}

\newtheorem{THEO}{Theorem}

\theoremstyle{remark}

\makeatletter
\newcommand{\bddots}{%
  \mathinner{\mkern1mu\raise\p@\vbox{\kern7\p@\hbox{.}}\mkern2mu
    \raise4\p@\hbox{.}\mkern2mu\raise7\p@\hbox{.}\mkern1mu}}
\makeatother
\def\i{{\mathtt{i}}}
 \def\a{{\alpha}}
 \def\b{{\beta}}

 \def\l{{\lambda}}

 \def\sb{{\mathbf s}}

 \def\CF{{\mathcal F}}

 \def\CJ{{\mathcal J}}
 \def\CL{{\mathcal L}}

 \def\CM{{\mathcal M}}
 \def\CS{{\mathcal S}}
 \def\CT{{\mathcal T}}

 \def\NN{{\mathbb N}}

 \def\RR{{\mathbb R}}

\newcommand{\wh}{\widehat}

\def\f{\frac}

\graphicspath{{./}}

\begin{document}

\title {Wronskians of Fourier and Laplace Transforms}

\author{Dimitar K. Dimitrov}
\address{Departamento de Matem\'atica Aplicada\\
 IBILCE, Universidade Estadual Paulista\\
 15054-000 Sa\~{o} Jos\'e do Rio Preto, SP, Brazil.}
 \email{dimitrov@ibilce.unesp.br}
\author{Yuan Xu}
\address{Department of Mathematics\\ University of Oregon\\
    Eugene, Oregon 97403-1222.}\email{yuan@uoregon.edu}

\date{\today}
\keywords{Fourier transform, Laplace transform, Wronskian, entire function, Laguerre-P\'olya class, Riemann hypothesis}
\subjclass[2010]{11M26, 30D10, 30D15, 40E05, 42A82}
\thanks{The first author is supported by the Brazilian foundations CNPq under Grant 307183/2013--0 and FAPESP under Grants 2016/09906--0 
and 2014/08328-8. 
The second author is supported in part by NSF Grant DMS-1510296}

\begin{abstract} 
Associated with a given suitable function, or a measure, on $\RR$, we introduce a correlation function, so that the Wronskian 
of the Fourier transform of the function is the Fourier transform of the corresponding correlation function, and the 
same holds for the Laplace transform. We obtain two types of results. First, we show that Wronskians of the Fourier 
transform of a nonnegative function on $\RR$ are positive definite functions and the Wronskians of the Laplace 
transform of a nonnegative function on $\RR_+$ are completely monotone functions. Then we establish necessary 
and sufficient conditions in order that a real entire function, defined as a Fourier transform of a positive kernel $K$,
belongs to the Laguerre-P\'olya class, which answers an old question of P\'olya. The characterization is given in 
terns of a density property of the correlation kernel related to $K$, via classical results of Laguerre 
and Jensen and employing Wiener's $L^1$ Tauberian theorem. As a consequence we provide a necessary and sufficient 
condition for the Riemann hypothesis in terms of a density of the translations of the correlation function related to 
the Riemann $\xi$-function.      
\end{abstract}

\maketitle

\section{Introduction}
\setcounter{equation}{0}

A real entire function $\varphi$ is in the Laguerre-P\'olya class, written $\varphi \in \mathcal{LP}$, if 
$$
   \varphi(z) = c z^m e^{-\a z^2+\b z} \prod_{k=1}^\infty (1+ z/x_k) e^{-z/x_k}
$$
for some $c,\b \in \RR$, $\a >0$, $m \in \NN_0$ and $x_k \in \RR \setminus\{0\}$, such that $\sum_k x_k^{-2} < \infty$. 
The class $\mathcal{LP}$ consists of entire functions that are uniform limits on the compact sets of the complex 
plane of polynomials with only real zeros. This class of functions was studied first by Laguerre in the ninetieth century 
and then more extensively by Jensen, P\'olya,  Schur, Obrechkoff and others in the beginning of the twentieth century 
because of the efforts towards the Riemann hypothesis. The latter connection is straightforward and we recall it very 
briefly. The Riemann $\xi$-function is defined in terms of the $\zeta$-function by (see \cite{Tit})
$$
\xi(s) = \frac{1}{2} s(s-1) \pi^{-s/2} \Gamma(s/2) \zeta(s).
$$
Define also $\Xi(z) = \xi(1/2+iz)$. The Riemann hypothesis states that $\Xi$, represented also as 
$$
\Xi(z) = \int_{-\infty}^{\infty} \Phi(u) e^{-izu} du,
$$
with
$$
\Phi(t) = 2 \sum_{n=1}^{\infty} (2 n^{4} \pi^{2} e^{9t/2} - 3 n^{2}
\pi e^{5t/2}) \exp (- n^{2} \pi e^{2t}),
$$
has only real zeros. Since $\Xi(z)$ is an entire function of order one, the Riemann hypothesis is equivalent to 
the fact that it belongs to $\mathcal{LP}$. Attempts to provide general tractable necessary and sufficient 
conditions for an entire function to be in $\mathcal{LP}$ had failed, so that in 1926 P\'olya \cite{Pol26} raised the question 
of characterizing the kernels $K$ whose Fourier transforms 
$$
    \int_{-\infty}^{\infty} K(u) e^{-izu} du
$$
belong to $\mathcal{LP}$.  We establish necessary and sufficient condition for P\'olya's problem for a subclass of Fourier transforms that 
contains $\Xi(z)$, which is given in terms of a density of a family of functions in $L^1(\mathbb{R})$. In order to formulate it, for a given real 
function $f$, we denote by $\CT(f)$ the span of its translations (or translates), that is,
$$
\CT(f) := \left\{ \sum_{k=1}^n c_k f(x+a_k),\ a_k \in \mathbb{R}, n \in \NN \right\}.
$$

Our result implies the following necessary and sufficient condition for the Riemann hypothesis: 

\begin{thm} \label{RH} The Riemann hypothesis is true if and only if, for each $y \in (-1/2,1/2)\setminus \{0\}$, 
the translates $\CT(\Phi_{2,y})$ of the kernel 
\begin{equation} \label{eq:RH}
\Phi_{2,y} (t) := \cosh(ty) \int_{-\infty}^{\infty}  (t-2s)^2\, \Phi(t-s)\, \Phi(s)\, ds 
\end{equation}
are dense in $L^1(\mathbb{R})$.

Furthermore, the translates $\CT(\Phi_{2,y})$ of $\Phi_{2,y} (t)$ are dense in $L^1(\mathbb{R})$ 
for every fixed $y \in (-1/2,1/2)$ if and only if the zeros of $\Xi(z)$ are real and simple. 
\end{thm}

It is worth mentioning that there are other density criteria for the Riemann hypothesis. We mention the classical 
Nyeman-Beurling criterion \cite{Beu,Nym} and its various generalizations and refinements due to B\'aez-Duarte and his 
collaborators \cite{BDc, BD}.  Other equivalent sufficient conditions for the Riemann hypothesis in terms of properties of 
the {\it correlation} kernel $\Phi_{2,y} (t)$ will be stated in the end of Section 3. 

The basic ingredients in the proof of Theorem \ref{RH} are classical results of Jensen and Laguere about entire 
functions in the Laguerre-P\'olya class, Wiener's  $L^1$ Tauberian theorem, known also as The Wiener Approximation 
Theorem, and a tool that we develop in this paper, called the correlation function associated to a given function or
a measure, which we now describe. 

Let $\CM(E)$ be the set of Borel measures on $E \in \RR$. For $\mu \in \CM(E)$ and $m=0,1,\ldots$, let 
$\mu_m: = \int_\RR t^m d\mu$ be the moment of $\mu$. Let $\CM_N(E)$ denote the set that consists of $\mu \in \CM(E)$ 
for which $\mu_n$ is finite for $n =0,1,\ldots N$. Furthermore, we denote by $\CM^+(E)$ and $\CM_N^+$ the subset
of  non-negative Borel measures, respectively. 

\begin{defn}
Let $\mu \in \CM_{2n-2}(\RR)$ be an absolutely continuous measure. For $n=2,3,\ldots$ we define a correlation function
\begin{equation}\label{eq:vn}
  \nu_n(t):= \nu_n(d\mu;t) = \int_{T^n(t)} \prod_{1 \le i< j \le n} (s_i-s_j)^2 \mu'(s_1)\ldots \mu'(s_n) d\sb,
\end{equation}
where $T^n(t)$ is the simplex in $\RR^n$ defined by
$$
T^n(t): = \{ (s_1,\ldots,s_n) \in \RR^n: s_1+\cdots + s_n =t\} 
$$
and $d\sb$ is the Lebesgue measure on $T^n(t)$. When $d\mu = w(t) dt$, we also write $\nu_n(w; t)$ 
and define $\nu_1(t): = w(t)$. 
\end{defn}

The correlation function is well--defined and is closely related to the Wronskian determinants of integral transforms. We 
study this function in view of the Fourier and Laplace transforms below. 

For $\mu \in \CM(\RR)$, let $\CF$ be the Fourier  transform of $\mu$ defined by, with $\i =\sqrt{-1}$, 
$$
   \CF \mu (x):= \wh \mu(x) = \int_\RR e^{-  \i x t }d\mu(t), \qquad x \in \RR.
$$
A measure $\mu \in \CM(\RR)$ is called {\it even} if $d\mu(t) = d\mu(-t)$. If $\mu $ is even, then $\CF \mu$ is a real valued 
function. Let $\RR_+ = [0, \infty)$. For $\mu \in \CM(\RR_+)$, let $\CL \mu$ be the Laplace transform of $\mu$ defined by 
$$
  \CL \mu(x):=  \int_\RR e^{-  x t }d\mu(t), \qquad x \in \RR.
$$

Let $f$ be a function in $C^{2n-2}(\RR)$, the class of functions that have continuous derivatives of order $2n-2$. The $n$-th Wronskian 
determinant of the function $f$ is defined by 
$$
W_n (f;x) :=    \det \left[f^{(i+j)}(x) \right]_{i,j =0}^{n-1}
 =\det \left[ \begin{matrix} f(x) & f'(x) & \ldots & f^{(n-1)}(x) \\
   f'(x) & f''(x) & \ldots & f^{(n)}(x) \\
   \cdots & \cdots & \cdots & \cdots \\
   f^{(n-1)}(x) & f^{(n)}(x) & \ldots & f^{(2n-2)}(x)
\end{matrix} \right] .
$$

The key ingredient, and our starting point, in this study is the observation that the Wronskian of the Fourier 
(or Laplace) transform of a function is the Fourier (respectively, Laplace) transform of the corresponding
correlation function. More precisely, we have the following theorem: 

\begin{thm} \label{thm:WronFL1}
For $n =2,3,\ldots$ and $f \in L^1(E)$ such that $\int_E |t^{2n-2} f(t)| dt < \infty$, 
\begin{enumerate} [ \quad \rm (1)]
\item $W_n(\CF f; \cdot) = (-1)^{n(n+1)/2}\CF (\nu_n(f))$ if $E = \RR$;
\item $W_n(\CL f; \cdot) = \CL ( \nu_n(f) )$ if  $E=\RR_+$.
\end{enumerate}
\end{thm}

This result allows us to prove, for example, that if $f$ is a nonnegative even function, then $(-1)^{n(n-1)/2} W_n(\CF f;x)$
is a strictly positive definite function on the real line, and $W_n(\CL f;x)$ is a completely monotone function on $\RR_+$. 

The paper is organised as follows. The next section is devoted to the study of the correlation functions, where we establish
the latter results. Examples that illustrate our results are also given in the section. The functions in the Laguerre-P\'olya class,
represented as a Fourier transform, are studied in Section 3, where we give the proof of Theorem 1.1. 
In Section 4 we provide some comments and results concerning Wiener's Tauberian theorem related to the main result established in 
Section 3. 

%For a correlation function $\nu_n(d\mu)$, we can use $d \nu_n(d\mu)$ in place of $d \mu$ and consider
%the Wronskin determinants of its Fourier or Laplace transform, which leads to an iteration process of the 
%correlation functions. This process will be discussed in Section 4. 

\section{Correlation functions and Wronskians of integral transforms}
\setcounter{equation}{0}

We discuss general properties of the correlation functions and its immediate applications in the first subsection. Several
examples are given in the second subsection. 

\subsection{Correlation functions and Wronskians}
The integral in the definition \eqref{eq:vn} has $n-1$ folds and it can be written explicitly so as an integral 
over the simplex $T^n(t) := \{s \in \RR^{n-1}: \a(s):= s_1+ \cdots +s_{n-1} \le t\}$ of $\RR^{n-1}$. Indeed, 
we can write the correlation function $\nu_n$ as 
\begin{align} \label{eq:vn-2}
    \nu_n(w;t) = \int_{T^n(t)} \prod_{1 \le i < j \le {n-1}}&  (s_i-s_j)^2    \prod_{i=1}^{n-1} (t- \a(s) -s_i)^2 \\
          &  \times \mu'(s_1)\cdots \mu'(s_{n-1}) \mu'(t-\a(s))ds_1\cdots ds_{n-1} \notag
\end{align}
by setting $s_n = t - \a(s) = t- s_1-\cdots - s_{n-1}$ in \eqref{eq:vn}. 

We need a classical identity that is the cornerstone of our results. The identity can be found in \cite[p. 62]{PS}.

\begin{lem} 
Let $f_i, g_i$, $1 \le i \le n$, be integrable functions defined on $\RR$ such that $f_i g_j \in L^1(\RR)$ for $1 \le i \ne j \le n$.
Then 
\begin{equation} \label{eq:polya-szego}
  \det \left[ \int_{\RR} f_i(t) g_j(t) d\mu(t) \right]_{i,j=1}^n = \int_{\RR^n} \det \left[ f_j(t_i) \right]_{i,j=1}^n 
  \det \left[ g_j(t_i) \right]_{i,j=1}^n \prod_{i=1}^n d\mu(t_i).
\end{equation}
\end{lem}

For $\mu \in \CM(E)$ and $n=0,1,\ldots$, let $M_n(d\mu)$ be the moment matrix defined by
$$
     M_n(d\mu): = \det \left[ \mu_{i+j} \right]_{i,j =0}^n.
$$
If $d\mu(t) =w(t) dt$, we write this determinant as $M_n(w)$. It is known that, if $\mu \in \CM_{2n}^+(E)$, then $M_n(d\mu)$ is positive definite and, in particular, $\det M_n(d\mu) > 0$. 

\begin{lem} \label{lem:int-mu}
Let $\mu \in \CM_{2n-2} (\RR)$ be absolutely continuous. Then the correlation function $\nu_n$ is integrable and 
$$
   \int_{\RR} \nu_n(d\mu; x) dx = \det M_{n-1}(d\mu). 
$$
Furthermore, if $d\mu(t) = w(t)dt$ and $w \in L^1(\RR)$, then $\nu_n \in L^1(\RR)$.
\end{lem}

\begin{proof}
Applying the previous lemma and using the Vandermond determinant 
$$
   V(t_1,\ldots,t_n) := \det \left[ t_i^{j-1} \right]_{i,j = 0}^{n-1} = \prod_{1 \le i< j \le n} (t_j - t_i),
$$
we obtain
\begin{align*}
  \int_{\RR} \nu_n(t) dt & = \int_\RR  \int_{\CS^n(t)} \prod_{1 \le i< j \le n} (s_i-s_j)^2 \mu'(s_1)\ldots \mu'(s_n)d \sb dt \\
   & = \int_{\RR^n}  V(s_1,\ldots,s_n)^2 d\mu(s_1)\ldots d\mu (s_n) \\
   &  = \det \left[ \int_\RR t^{i+j} d\mu(t) \right]_{i,j =0}^{n-1} = \det M_{n-1}(d\mu), 
\end{align*}
where we have used \eqref{eq:polya-szego} with $f_j(t) = g_j(t) = t^j$. 

If $d\mu(t) =w(t) dt$ and $w\in L^1(\RR)$, then the same proof shows that 
$$
    \int_{\RR} |\nu_n(t)| dt \le \int_{\RR^n}  V(s_1,\ldots,s_n)^2 w(s_1)\ldots w(s_n) ds_1\cdots ds_n = \det M_{n-1}(|w|),
$$
which is finite and positive, so that $\nu_n \in L^1(\RR)$. 
\end{proof}
 
\begin{lem}
If $d\mu$ is an even measure, then $\nu_n$ is an even function for $n =2,3,\ldots$. 
\end{lem}

\begin{proof}
Changing variables $s_i \to - s_i$ in the  definition of $\nu_n$ shows that $\nu_n(t) = \nu_n(-t)$ if $d\mu$ is 
symmetric with respect to the origin. 
\end{proof}
 
\begin{lem}
If $d\mu(t) =f(t)dt $ and $f$ is supported on the interval $(a,b)$, then $\nu_n(f)$ is supported on $(na, nb)$. 
\end{lem}

\begin{proof}
If $f$ is supported on $(a,b)$, then $a \le s_i \le b$ in \eqref{eq:vn-2}. As a result, if $t \le na$, then
$t - s_1-\cdots -s_{n-1} \le a$, so that $f(t- \a(s)) =0$. Similarly, if $t \ge nb$, then $f(t-\a(s)) =0$. 
\end{proof}

We are now ready to prove Theorem \ref{thm:WronFL1}, which we restate below: 

\begin{thm} \label{thm:WronFL}
For $n =2,3,\ldots$, $f \in L^1(E)$ such that $t^{2n-2} f \in L^1(E)$, 
\begin{enumerate} [ \quad \rm (1)]
\item $W_n(\CF f; \cdot) = (-1)^{n(n+1)/2}\CF (\nu_n(f))$ if $E= \RR$;
\item $W_n(\CL f; \cdot) = \CL ( \nu_n(f) )$ if  $E =\RR_+$.
\end{enumerate}
\end{thm}

\begin{proof}
We prove (2) first. Let $d\mu = f (t) dt$. The condition $t^{2n-2} f \in L^1(\RR_+)$ implies that the derivatives of the 
Laplace transform $\CL f$, up to $2n-2$ order, are well defined continuous functions, so is $W_n(\CL f;\cdot)$. 
Applying \eqref{eq:polya-szego} with $f_j(t) = g_j(t) = t^j$, we obtain 
\begin{align*} 
  W_n(\CL f;x) &  = \det \left[ \int_{\RR_+} (-t)^{i+j} e^{-t x} d \mu(t) \right]_{i,j =0 }^{n-1} \\
    &  = \int_{\RR_+^n} \prod_{1 \le i < j \le n} (s_i- s_j)^2 e^{- x(s_1+\cdots + s_n)} d\mu(s_1)\cdots d\mu(s_n) \\
   & =  \int_\RR \nu_n(f;t) e^{-t x} dt = \CL \nu_n(f;x),
\end{align*}
which proves (2). The proof of (1) is similar. Taking the derivatives of $\CF \mu$ introduces powers of $\i$. Using 
$$
  \prod_{1 \le i < j \le n} (- \i s_j + \i s_i)^2 = \i^{n(n-1)/2} \prod_{1 \le i < j \le n} (s_i-s_j)^2 
   = (-1)^{n(n-1)/2} \prod_{1 \le i < j \le n} (s_i-s_j)^2,
$$
the proof follows as in the case (2). 
\end{proof}

Although these are simple relations, we give two nontrivial applications below to show that they are not as
obvious as they appear to be.
 
Recall that a function $\psi: \RR \mapsto \RR$ is called {\it positive definite} if 
$$
   \sum_{i =1}^N \sum_{j=1}^N c_i c_j \psi (x_i - x_j) \ge 0
$$
for all $x_1\ldots, x_N \in \RR$ and $N =1,2,\ldots$, and it is called {\it strictly positive definite} if $\ge 0$ is 
replaced by $>0$ whenever $(c_1,\ldots,c_N)$ is not identically zero. The positive definite functions are 
characterized by Bochner's theorem: a continuous function $\psi$ is positive definite if and only if it is the 
Fourier transform of a finite non-negative Borel measure. Together with Theorem \ref{thm:WronFL}, we can
then state the following theorem: 

\begin{thm} \label{thm:WronF}
For each $n \in \NN$, let $\mu$ be an even measure in $\CM_{2n-1}^+(\RR)$. Then the multiple of the Wronskian 
determinant 
 \begin{equation}\label{eq:WronF}
    W_n^\CF(\mu; x):= (-1)^{n(n-1)/2} W_n(\CF\mu;x)
\end{equation}
is a strictly positive definite function on the real line. 
\end{thm}
 
We only have to comment on the strictly positiveness of the statement, which is not covered by Bochner's theorem on 
the positive definite functions. However, it is known that if $\psi$ is the Fourier transform of a finite non-negative Borel 
measure and the measure is not discrete, then $\psi$ is strictly positive.  

A function $\phi: \RR_+ \mapsto \RR$ in $C(\RR_+) \cap C^\infty (\RR_+)$ is called a {\it completely monotone function}
if it satisfies 
$$
    (-1)^k \phi^{(k)} (x) \ge 0, \qquad x > 0, \quad k =0,1,2,\ldots. 
$$
The completely monotone functions are characterized by Bernstein's theorem: a function is completely monotone 
if and only if it is the Laplace transform of a finite non-negative Borel measure $\mu$ on $\RR_+$. Together with 
Theorem \ref{thm:WronFL}, we can then state the following theorem, which appeared first in \cite[Corollary 4.6]{DX} 
and motivated our study in the present paper: 

\begin{thm} \label{thm:WronL}
Let $n \in \NN$ and $\mu \in \CM_{2n-2}^+(\RR_+)$. For each $n \in \NN$, $W_n(\CL\mu; \cdot)$ is a completely 
monotone function on $\RR_+$. 
\end{thm}

We state a corollary of Theorem \ref{thm:WronF} and Theorem \ref{thm:WronL}.
Let $(\CF \mu)^{(k)}$ and $(\CL \mu)^{(k)}$ denote the $k$-th order derivative of $\CF \mu$ and $\CL \mu$,
respectively. 

\begin{thm}
Let $k =1,2,\ldots$ and $n=2,3,\ldots$. Then 
\begin{enumerate} [ \quad \rm (1)]
\item If  $\mu \in \CM_{k+n-2}^+\RR)$ is even, then $(-1)^{n k/2} W_n( (\CF \mu)^{(k)}; \cdot)$ is a strictly positive definite
function on the real line provided either $k$ or $n$ is even. 
\item If  $\mu \in \CM_{k+n-2}^+(\RR_+)$, then $(-1)^{n k} W_n((\CL\mu)^{(k)}; \cdot)$ is a completely monotone function
on the real line.  
\end{enumerate}
\end{thm}

\begin{proof}
For $k=1,2,\ldots$, let $\{\cdot\}^k \mu$ be the measure defined by $d (\{\cdot\}^k \mu)(x) := x^k d\mu(x)$.
Since $(\CF \mu)^{(k)}(x) = (- \i)^k \CF (\{\cdot\}^k \mu; x)$, it is easy to see that 
$$
   W_n \left((\CF \mu)^{(k)}; x \right) = (-\i)^{n(n+k+1)} W_n \left( \CF (\{\cdot\}^k \mu); x \right),
$$
which is real valued and $(-\i)^{n(n+k+1)} =  (-1)^{n(n+k+1)/2}$ if $n k $ is even. Since 
$\{\cdot\}^k \mu \in \CM_{n-2}^+(\RR)$, (1) follows from Theorem \ref{thm:WronF}. Similarly, (2) follows as 
a corollary of Theorem \ref{thm:WronL}.
\end{proof}
 
In particular, in the case of $n =2$, this shows that 
$$
     W_2 \left((\CF \mu)^{(k)}; x \right) =(-1)^k \left( [(\CF \mu)^{(k+1)}(x)]^2 - (\CF \mu)^{(k)}(x)(\CF \mu)^{(k+2)}(x) \right)
$$
is a strictly positive definite function. For a nontrivial example of such results, we refer to Corollary \ref{cor:Bessel} in
the next subsection.

\subsection{Examples}
We give several examples to illustrate our results. First we point out that, by Theorem \ref{thm:WronL} and
Theorem \ref{thm:WronF}, 
\begin{equation}\label{eq:nu=InvW}
  \nu_n(d\mu; x) = \CL^{-1} [W_n (\CL(d\mu); \cdot)](x) \quad \hbox{and}\quad
       \nu_n(d\mu; x) = \CF^{-1} [W_n^\CF(d\mu; \cdot)](x), 
\end{equation}
where we assume that $d\mu$ is supported on $\RR_+$ in the first identity. In general, taking the Fourier or 
Laplace transform of $W_n^\CF (d\mu; \cdot)$ or $W_n  (\CL(d\mu); \cdot)$, respectively, is difficult and the 
above identity may not be useful for determining the explicit formula of $\nu_n$. In some cases, however, it 
can be used as shown in our first two examples. 

\begin{prop}
If $d\mu (t) = e^{-t^2/2} dt $, then 
$$
  \nu_n(d\mu; t) = a_n e^{-t^2/{2n}} \quad \hbox{with}\quad a_n = \frac{1}
  {\sqrt{n}} (2\pi)^{(n-1)/2} \prod_{k=1}^{n-1}k!. 
$$
In particular, the span of $\{e^{-(t-a)^2/{2n}}: a \in \RR\}$ is dense in $L^1(\RR)$.  
\end{prop}

\begin{proof}
Let $h(t) =  e^{-t^2/2}$. It is well-known that $\wh h (t) = \sqrt{2 \pi} e^{-t^2}$. Furthermore, by the Rodrigues 
formula of the Hermite polynomials, it is easy to see that 
$$
 \f{d^n}{dx^n} \wh h(x) = \sqrt{2\pi} 2^{-n/2} (-1)^n H_n(x/2).
$$
The closed formula of the Wronskian of the Hermite polynomials is known; see, for example, 
the identity \cite[(33)]{Lec}, from which we deduce that  
$$
   W_n(\CF h; x) = \left( (-1)^{n(n-1)/2} (2 \pi)^{n/2} \prod_{k=1}^{n-1} k! \right) e^{-n x^2/2}.
$$
Taking the inverse Fourier transform of this identity, the formula for $\nu_n(t)$ follows from
(1) in Theorem \ref{thm:WronFL}.
\end{proof}

\begin{prop}
Let $d\mu_\a(t) = t^\a e^{-t} dt$, $\a > -1$, be supported on $\RR_+$. Then 
$$
  \nu_n(d\mu_\a;t) = a_n t^{ n(n+\a)-1} e^{-t} \quad\hbox{with} \quad a_n =   
   \frac{\Gamma(\a+1)^n}{\Gamma(n(n+\a))}  \prod_{k=1}^{n-1} k! (\a+1)_k.
$$
\end{prop}

\begin{proof}
Let $h_\a(t) = t^\a e^{-t}$ on $\RR_+$.  The Laplace transform of $h_\a$ is given by 
$\CL h_\a (t) = \Gamma(\a+1) (1+t)^{-\a-1}$, so that $(\CL h_\a (t))^{(k)} = (-1)^k 
 \Gamma(\a+k) (1+t)^{-\a- k}$. It follows that 
$$
  W_n( \CL h_\a; x) =  \left(\prod_{k=1}^{n-1} k! (\a+1)_k \right) \frac{\Gamma(\a+1)^n}{(1+x)^{n (n+\a)}}.
$$
Indeed, the power of $1+x$ can be completely factored out from the determinant, so that
$W_n(\CL h_\a;x) = c (1+x)^{-n (n+\a)}$ and the constant determinant can be evaluated to the value given. 
Taking the inverse Laplace transform of this identity, the formula of $\nu_n(t)$ follows from (2) in 
Theorem \ref{thm:WronFL}.
\end{proof}

Our next example uses the definition of the correlation function, or rather \eqref{eq:vn-2}, to derive an explicit 
formula for $\nu_2$ for a family of functions. Let ${}_2F_1$ be the standard Gauss hypergeometric function. 

\begin{prop}\label{prop:Jacobi-correlation}
For $\a, \b  > -1$, let $w_{\a,\b}(t) = (1-t)^\a(1+t)^\b  \chi_{[-1,1]}(t)$. Then $\nu_2(w_{\a,\b}; t) =0$
if $|t| > 2$ and 
$$
  \nu_2(w_{\a,\b} ;t) = \frac{\Gamma(\f32) \Gamma(\a+1)}{2^{2 \a+1} \Gamma(\a+\f{5}2)}
     (2-t)^{2\a+3} (2 t)^{\b} {}_2F_1\left( \begin{matrix} - \b, \a+1 \\ \a+\frac{5}{2} \end{matrix}; - \frac{(2-t)^2}{8t}\right)
$$
if $0 \le t \le 2$, and $\nu_2(w_{\a,\b};- t) = \nu_2(w_{\b,\a};t)$. 
\end{prop}

\begin{proof}
Directly from the expression of $\nu_2$ at \eqref{eq:vn-2}, we obtain that 
$$
  \nu_2(t) = \int_{-\infty}^{\infty} (2s-t)^2 w(s) w(t-s) ds,
$$
from which it follows immediately that $\nu_2(w_{\a,\b};t) = 0$ if $|t| > 2$ and $\nu_2(w_{\a,\b};- t) = \nu_2(w_{\b,\a};t)$. 
Furthermore, for $0 \le t \le 2$, changing variable $u = -1+2 (1-s)/(2-t)$ gives
\begin{align*}
  \nu_2(w_{\a,\b};t) & =  \int_{-1+t}^{1} (2s-t)^2 w_{\a,\b}(s) w_{\a,\b}(t-s) ds \\
      & = \frac{(2-t)^{2 \a+3}}{2^{ 2 \a +1}} \int_{-1}^1 \left( 2 t + \f{(2-t)^2}{4} (1-u^2)\right)^\b u^2 (1-u^2)^\a du \\
      & =\frac{(2-t)^{2 \a+3}}{2^{ 2 \a +1}} (2 t)^\b \int_{-1}^1 \left( 1 + \frac{(2-t)^2}{8 t}(1-u^2)\right)^{\b}
         u^2 (1-u^2)^{\a} du.
\end{align*} 
Let $s = - (2-t^2)/(8t)$. Expanding $(1 + s (1-u^2))^{\b}$ in infinite series, we obtain
\begin{align*}
  \int_{-1}^1 \left( 1 - s (1-u^2)\right)^{\b} u^2 (1-u^2)^{\a} du 
 &  = \sum_{n=0}^\infty \frac{(-\b)_n}{n!}\int_{-1}^1 u^2 (1-u^2)^{n+\a} du s^n \\
   &  = \frac{\Gamma(3/2) \Gamma(\a+1)}{\Gamma(\a+\f52)}
      \sum_{n=0}^\infty \frac{(-\b)_n (\a+1)_n }{(\a+\f52)_n n!} s^n,
\end{align*}
the last summation is the ${}_2F_1$ function. This completes the proof. 
\end{proof}

\begin{cor}\label{oml}
For $\l > -1/2$, let $w_\l(t) := (1-t^2)^{\l-1/2} \chi_{[-1,1]}(t)$. Then 
\begin{equation}\label{nuoml}
  \nu_2(w_\l;t) = \frac{\Gamma(\f32) \Gamma(\l+\f12)}{2^{\l+\f12} \Gamma(\l+2)}
     (2-|t|)^{2\l+2} |t|^{\l-\f12} {}_2F_1\left( \begin{matrix} - \l+ \frac12, \l+\f12 \\ \l+2 \end{matrix}; \frac{(2-|t|)^2}{8|t|}\right)
     \chi_{[-2,2]}(t).
\end{equation}
In particular, the above ${}_2F_1$ function is nonnegative on the interval $[-2,2]$.
\end{cor}
  
The Fourier transform of $w_{\a,\b}$ is given by 
$$
    \CF w_{\a,\b} (x) = \frac{2^{\a+\b+1}\Gamma(\a+1)\Gamma(\b+1)}{\Gamma(\a+\b+2)} 
      e^{ \i t} {}_1F_1\left( \begin{matrix} \b+1 \\ \a+\b+2 \end{matrix}; -2 \i t\right). 
$$
More interestingly, the Fourier transform of $w_\l$ can be expressed in terms of the Bessel function $J_\l(x)$. 
Indeed, by the integral representation of the Bessel function, 
$$
  \CF w_\l(x) = \int_{-1}^1 e^{-\i tx} (1-t^2)^{\l-1/2} dt =   \sqrt{\pi}\Gamma(\l+1/2) \left(\frac{2}{x}\right)^\l J_\l(x).
$$
Using the identity $J_\l'(x) = (J_{\l-1}(x) - J_{\l+1}(x))/2$, it follows that 
\begin{align*}
  - W_2(\CF w_\l;x) =\, & 4^\l\pi \Gamma(\l+\tfrac12)^2 x^{-2(1+\l)} \\
    &  \times \left[ x^2 J_{\l -1}(x)^2 - (2 \l-1) x J_{\l+1}(x)J_\l(x)+ (x^2 -2\l)J_\l(x)^2\right].
\end{align*}
By Theorem \ref{thm:WronFL}, this function is equal to the Fourier transform of $\nu_2(w_\l;t)$, which is
non-trivial since a direct verification is not immediate and, in fact, looks formidable for a generic parameter
$\l$. In the case of $\l = 1/2$, $w_{1/2}$ is the characteristic function $w_{1/2}(t) = \chi_{(-1,1)}(t)$. In this case,
$$
  \CF w_{1/2}(t) = 2 \frac{\sin t}{t} \quad \hbox{and} \quad \nu_2(w_{1/2};t) = \f{1}{3} (2-|t|)^3. 
$$ 
 
It is known that if $\phi$ is a positive definite function, then 
$$
     \phi(0) > 0 \quad \hbox{and} \quad   |\phi(x)| \le \phi(0), \quad \forall x\in \RR.    
$$
However, a positive definite function is not necessarily positive everywhere. 

\begin{cor} \label{cor:Bessel}
For $\l > -1/2$, let $\CJ_\l(x):= \left(\frac{1}{t}\right)^\l J_\l(t)$. For $n = 2, 3,\ldots$, the Wronskian 
determinant $W_n^{(\l)} (x): =(-1)^{n(n-1)/2} W_n(\CJ_\l; x)$ is a strictly positive definite function on $\RR$
and $W_n^{(\l)} (x) \le W_n^{(\l)} (0)$ for $x \in \RR$. 
\end{cor}

Even in the case of $w_0(t) = \chi_{[-1,1]}(t)$, that the Wronskian $W_n^\CF(x)$ leads to a positive definite 
function appears to be nontrivial. For example, the case $n =2$ shows that 
$$
   W_2^\CF(x)= 2 \frac{-1 + 2 x^2 + \cos (2 x)}{x^4}
$$
is a strictly positive function on the real line and $W_2^\CF(x) \le 4/3$ for all $x\in \RR$.

As an application of Proposition \ref{prop:density} in Section 4, $W_2^\l(x) > 0$ on $\RR$. In this case, the inequality in 
Corollary \ref{cor:Bessel} gives: 

\begin{cor} 
For $\l > -1/2$ and $x \in \RR$,  
\begin{equation*}
  0 <  (\CJ_\l'(x))^2 - \CJ_\l(x) \CJ''_\l (x) \le \frac{1}{2} \left(\frac{1}{\Gamma(\l+1)^2} -  
     \frac{1}{\Gamma(\l)\Gamma(\l+2)}  \right).
\end{equation*}
In particular, the right hand side becomes $1/2$ when $\l = 0$.
\end{cor}

We give another example, where the ``tent function'' $w_\Lambda$ is such that its Fourier transform is the function
$(\sin t/t)^2$. 

\begin{prop}\label{tent}
Let $w_\Lambda(t) = \frac{1}{4} (2-|t|)_+$, where $x_+ = x$ if $x \ge 0$ and $x_+ = 0$ otherwise. Then
the correlation function $\nu_2(w_\Lambda;\cdot)$ is given by
\begin{equation}
\label{omL}
 \nu_2(w_\Lambda; t) =  \frac{1}{480} \chi_{[-4,4]}(t) 
     \begin{cases} (4-|t|)^5 -2 (2-|t|)^5 - 40 (2-|t|)^3, &  0 \le t \le 2, \\ 
       (4-|t|)^5, &  t \ge 2.  
     \end{cases} 
\end{equation}
\end{prop}

\begin{proof}
As in the proof of Proposition \ref{prop:Jacobi-correlation}, it is easy to see that $\nu_2(w_\Lambda)$ is supported
on $[-4,4]$ and it is an even function. Assume $0 \le t \le 4$. It follows then that 
$$
  \nu_2(w_\Lambda;t) = \frac{1}{16}\int_{-2+t}^2 (2s -t)^2 (2 - s) (2 - |t - s|) ds. 
$$
Evaluating the integral gives the stated result. 
\end{proof}

For the Laplace transform, we need $d\mu$ supported on $\RR_+$. We give one example. 

\begin{prop}
For $\a, \b  > -1$, let $u_{\a,\b}(t) = t^\a(1-t)^\b  \chi_{[0,1]}(t)$. Then $\nu_2(u_{\a,\b}; t) =0$
if $t > 2$, 
$$
  \nu_2(u_{\a,\b} ;t) = \frac{\Gamma(\f32) \Gamma(\a+1)}{2^{2 \a+1} \Gamma(\a+\f{5}2)}
     t^{2\a+3} (1-t)^{\b} {}_2F_1\left( \begin{matrix} - \b, \a+1 \\ \a+\frac{5}{2} \end{matrix};  \frac{t^2}{4(1-t)}
     \right)
$$
if $0 \le t \le 1$, and 
$$
  \nu_2(u_{\a,\b} ;t) = \frac{\Gamma(\f32) \Gamma(\b+1)}{2^{2 \b+1} \Gamma(\b+\f{5}2)}
     (2-t)^{2\b+3} (t-1)^{\a} {}_2F_1\left( \begin{matrix} - \a, \b+1 \\ \b+\frac{5}{2} \end{matrix};  \frac{(2-t)^2}{4(1-t)}
     \right)
$$
if $1 \le t \le 2$. 
\end{prop}

\begin{proof}
Since $u_{\a,\b} (\frac{1-t}2) = w_{\a,\b}(t) /2^{\a+\b}$, a simple change variable shows that 
\begin{align*}
  \nu_2(u_{\a,\b};t) & =  2^{-2\a-2\b-1} \int_{u_1+u_2 = 2-2t} w_{\a,\b}(u_1)w_{\a,\b}(u_2)du_1du_2 \\
   & = 2^{-2\a-2\b-1} \nu_2(w_{\a,\b}; 2(1-t)),
\end{align*}
from which the stated formula follows from the previous proposition.  
\end{proof}

For $\a,\b > -1$, the Laplace transform of $u_{\a,\b}$ is given by 
$$
  \CL u_{\a,\b} (x) = \frac{\Gamma(\a+1)\Gamma(\b+1)}{\Gamma(\a+\b+2)}
     {}_1F_1\left( \begin{matrix} \a+1 \\ \a+\b+2 \end{matrix}; -t\right). 
$$ 
According to Theorem \ref{thm:WronL}, the Wronskian of this function is a completely monotone function. 
In the case of $\a = \b = \l-1/2$, we write $u_\l(t) = (t (1-t))^{\l-1/2}$ and the Laplace transform is given in
terms of the modified Bessel function $I_n(t)$ of the first kind,  
$$
     \CL u_{\l} (t) = \sqrt{\pi}\Gamma(a+\tfrac{1}{2}) t^{-a} e^{-t} I_n\left(\f t 2 \right) e^{-t},
$$
where $I_n(t) = \i^n J_n(\i t)$ is real valued. In particular, in the case of $\a = \b =0$, we have 
$$
      \CL \chi_{[0,1]} (t) = \frac{1 - e^{-t}}{t},
$$
which is a completely monotone function. The Wronskin of this function is also completely monotone. The 
simplest case $n=2$ shows that 
$$
   W_2^\CL(t) = \frac{2 \cosh(t) - t^2 -2}{t^4} e^{-t} 
$$
is completely monotone.

\section{Fourier transforms in the Laguerre-Polya class}
\setcounter{equation}{0}

In this section we establish the necessary and sufficient conditions for an entire function, represented as a 
Fourier transform of an absolutely continuous Borel measure, to belong to the Laguerre-P\'olya class in terms 
of density of the corresponding correlation functions. In fact, we consider an additional assumption on the entire function 
that is satisfied by the Riemann $\Xi$-function. As it has been mentioned, our result provides an answer 
of the problem of P\'olya \cite{Pol26} for this specific subclass. 

Recall that a real entire function $\varphi$ is in the Laguerre-P\'olya class, $\varphi \in \mathcal{LP}$, if 
$$
   \varphi(z) = c z^m e^{-\a z^2+\b z} \prod_{k=1}^\infty (1+ z/x_k) e^{-z/x_k}
$$
for some $c,\b \in \RR$, $\a >0$, $m \in \NN_0$ and $x_k \in \RR \setminus\{0\}$ such that $\sum_k x_k^{-2} < \infty$. 

Already Laguerre observed that if $\varphi \in \mathcal{LP}$, then 
\begin{equation}\label{LIG}
[\varphi^{(j)}(x)]^2 - \varphi^{(j-1)}(x) \varphi^{(j+1)}(x) \geq 0,\ \ x\in \mathbb{R},\ \ j\in \mathbb{N}.
\end{equation}
Although the Laguerre inequalities (\ref{LIG}) consist of an infinite set of conditions, they are only necessary for an 
entire function $\varphi$ to belong to the Laguerre-P\'olya class. In fact, (\ref{LIG}) follow from the most simple Laguerre inequalities 
\begin{equation}\label{LI}
[\varphi^{\prime}(x)]^2 - \varphi(x) \varphi^{\prime\prime}(x) \geq 0,\ \ x\in \mathbb{R}
\end{equation}
and the fact that $\mathcal{LP}$ is closed under differentiation. 
In 1913 Jensen \cite{Jen} made the ingenious 
observation that $\varphi$ should have only real zeros if $|\varphi(z)|^2$ is either an increasing function along all rays 
perpendicular to the real line or convex along all such lines. Let $\varphi$ be a real entire function
\begin{equation} \label{fiRI}
 \varphi(z) =  \phi(x,y) + \i \psi(x,y),\qquad z=x+\i y,
\end{equation} 
whose real and imaginary parts are 
\begin{align*}
\phi(x,y): =   \Re \varphi(z) =  \frac12 \left(\varphi(z) + \varphi({\bar z}) \right) \quad \hbox{and} \quad
\psi(x,y): =   \Im \varphi(z) =  \frac1{2 \i} \left(\varphi(z) - \varphi({\bar z}) \right).
\end{align*}
Jensen's criteria state that, under a mild additional restriction, the zeros of $\varphi$ should be real
if and only if 
\begin{equation} \label{fi2}
   | \varphi(z)|^2=[\phi(x,y)]^2 + [\psi(x,y)]^2
\end{equation}
is either an increasing function of $y\in [0,\infty)$ or it is a convex function of $y\in  (-\infty,\infty)$. He observed that 
$$
\frac{1}{2} \frac{\partial^2}{\partial y^2}    | \varphi(z)|^2 = | \varphi^\prime(z)|^2 - \Re (\varphi(z) \overline{\varphi^{\prime\prime}(z)}).
$$
Jensen's convexity criterion states:

\begin{THEO} 
\label{JT} 
Let $\varphi(z)=e^{-az^2} \varphi_1(z)$, $a\geq 0$, $\varphi \not\equiv 0$, where $\varphi_1$ is a real entire function of 
genus $0$ or $1$. Then $\varphi \in \mathcal{LP}$ if and only if 
\begin{equation} 
\label{Ji}
| \varphi^\prime(z)|^2 - \Re (\varphi(z) \overline{\varphi^{\prime\prime}(z)}) \geq 0\ \ \mathrm{for\ all}\ z\in \mathbb{C}.
\end{equation}
\end{THEO}

Jensen's result was refined recently by Csordas and Vishnyakova \cite{CsoVish} who proved that 
if $\varphi(z)$ is a real entire function, $\varphi \not\equiv 0$, and Jensen's inequalities (\ref{Ji}) hold
 then $\varphi \in \mathcal{LP}$.

However, if one differentiates the right-hand side of (\ref{fi2}) twice with respect to $y$, uses the fact that both 
$\phi(x,y)$ and $\psi(x,y)$  are harmonic functions and applies the Cauchy-Riemann 
equations, obtains (see also \cite{Cso15, CsoVish})
\begin{align*}
|\varphi^\prime(z)|^2 - \Re (\varphi(z) \overline{\varphi^{\prime\prime}(z)}) = & \ [\phi_x (x,y)]^2 - \phi(x,y) \phi_{xx}(x,y)\\
\ &  +  [\psi_x (x,y)]^2 - \psi(x,y) \psi_{xx}(x,y)\\
 = & \ - \left[ W_2 (\phi(\cdot,y);x) + W_2 (\psi(\cdot,y);x) \right].
\end{align*} 
Therefore, Jensen's result can be rewritten in the following form:
 
\begin{cor} 
\label{WCor}
Let $\varphi(z)$, defined by {\rm (\ref{fiRI})}, be a real entire function that obeys the requirements in Theorem \ref{JT}. Then
$\varphi \in \mathcal{LP}$ if and only if 
\begin{equation}\label{WrLP}
W_2 (\phi(\cdot,y);x) +  W_2 (\psi(\cdot,y);x) \leq 0 \ \ \ \mathrm{for\ all}\ z=x+\i y\in \mathbb{C}.
\end{equation}
\end{cor}

Let $K$ be an even function that decreases rapidly enough at infinity, so that 
\begin{equation}
\label{fiFT}  
 \varphi (z) :=  \int_{-\infty}^\infty K(u) e^{-i u z} du
\end{equation}
is an entire function of the form described in  Theorem \ref{JT}. 
Then its real and imaginary part are given by 
$$
  \phi(x, y) = \int_\RR \cosh(s y) K(s) \cos (s x) ds, \quad  \psi(x, y) = \int_\RR \sinh(s y) K(s) \sin (s x) ds. 
$$

The observations presented in this section up to now are classical and they are due to Jensen \cite{Jen}, 
P\'olya \cite{Pol27} and developed further by Csordas and Varga \cite{CsoVar}. Roughly speaking, they say 
that $\varphi \in \mathcal{LP}$ if and only if the sum of the Laguerre quantities (\ref{LI}) for the real and 
imaginary part of $\varphi$, considered as functions of $x$, are nonnegative for any 
fixed $y$.

In what follows we restrict our study to a narrower class of entire function which obey specific properties 
that are verified for the Riemann $\Xi(z)$. The main features we employ are the facts that the zeros of $\Xi$ lie in the horizontal
strip $S_{1/2}= \{ z\, :\, |\Im z|<1/2\}$ because the nontrivial zeros of $\zeta$ lie in the critical strip and it possesses the Hadamard factorization 
$$
\Xi(z) = \Xi(0) \prod_{k=1}^\infty \left( 1 - \frac{z}{z_k} \right),\ \ z_k \in S_{1/2}.
$$ 
It is worth mentioning that Riemann stated the latter representation without proof in \cite{R}  and it was established rigorously by 
Hadamard. 

Therefore, in what follows, we consider real entire functions $\varphi$ with the following properties:
\begin{itemize} 
\item[(i)] $\varphi(z)$ is represented as a Fourier trasform (\ref{fiFT}) of an even kernel $K$ in the Schwartz space $\mathcal{S}(\RR)$
(see \cite{SS});
\item[(ii)] The zeros of $\varphi$ belong to the horizontal strip  
$$
S_\a = \{z=x+iy: - \a < y < \a\}
$$
with width $2\a > 0$;
\item[(iii)] The function $\varphi$ possesses the Hadamard factorization 
\begin{equation}
\label{HF}
\varphi(z) = \varphi(0) \prod_{k=1}^\infty \left( 1 - \frac{z}{z_k} \right),\ \ z_k \in S_\a.
\end{equation}
\end{itemize}

We shall prove that an entire function which obeys these requirements belongs to the Laguerre-P\'olya 
class if and only if it satisfies the {\em strict} inequalities (\ref{Ji}), or equivalently the {\em strict} inequalities 
(\ref{WrLP}), for every $y\neq 0$.  We shall employ the results from the previous sections and the 
celebrated Wiener's $L^1$ Tauberian Theorem \cite[Theorem II]{Wien}:
\begin{THEO}
\label{W1}
If $K\in L^1(\RR)$, a necessary and sufficient condition for the set of all its translations  to be dense in 
$L^1(\RR)$ is that its Fourier transform $\CF(K)$ should have no real zeros.
\end{THEO} 
It is worth mentioning that Wiener's $L^2$ Tauberian Theorem \cite[Theorem I]{Wien} states that 
the translates of $K\in L^1(\RR) \cap L^2(\RR)$ are dense in $L^2(\RR)$ if and and only if 
its Fourier transform, extended to $L^2(\RR)$  via Plancherel's theorem, is such that  
its real zeros should form a set of zero measure.

Thus, the main result in this section reads as follows:

\begin{thm}\label{MT3}
Let  the real entire function $\varphi$ satisfy the above properties {\rm (i)}, {\rm (ii)} and {\rm(iii)}.  
Then $\varphi \in \mathcal{LP}$ if and only if, for every fixed $y\in (-\alpha,\alpha)\setminus \{0\}$,
the translates $\CT (K_{2,y})$ of the kernel 
$$
K_{2,y} (t) = \cosh(ty) \int_{-\infty}^{\infty}  (t-2s)^2 K(t-s) K(s)\, ds
$$
are dense in $L^1(\mathbb{R})$.

Furthermore, the translates $\CT (K_{2,y})$ of $K_{2,y} (t)$ are dense in $L^1(\mathbb{R})$ for every fixed 
$y\in (-\alpha,\alpha)$ if and only if the zeros of $\varphi$ are real and simple.  
\end{thm}

We shall need the following simple technical result:

\begin{lem}
\label{fipsi}
Let $\varphi$ be defined by {\rm (\ref{fiFT})}. For a fixed $y$, let $\mu_y^\phi (s):= \cosh(s y) K(s)$ and 
$\mu_y^\psi (s):= \sinh(s y) K(s)$. Then 
$$
  \phi(x,y) = (\CF \mu_y^\phi)(x) \quad \hbox{and} \quad \psi(x,y) = (\CF \mu_y^\psi)(x).
$$
\end{lem}

\begin{proof}
Since $K$ is even, changing variable $s \mapsto -s$ shows that 
$$
    (\CF \mu_y^\phi) (x) = \int_{\RR}  \cosh(s y) K(s) e^{-i s x} ds = \int_{\RR}  \cosh(s y) K(s) \cos(s x) ds 
        = \phi(x,y)
$$
The case of $\CF \mu_y^\psi$ follows similarly. 
\end{proof}

We shall need also some results concerning entire functions in the Hermite--Biehler  class (see \cite[Chapter 7]{Lev}). 
The entire  function $\omega$ is said to be in this class, denoted $\omega \in HB$, if it has no zeros in the closed 
lower half-plane $\Im z \leq 0$ and $|\omega(z)/\bar{\omega}(z)| < 1$ for every $z$ with $\Im z >0$. Here 
$\bar{\omega}(z)=\overline{\omega(\bar{z})}$. Let $\omega(z) = P(z) + \i Q(z)$, where 
the real and imaginary parts $P$ and $Q$ are real entire functions. An important fact we shall need is the following 
(see \cite[Theorem 3 on pp. 311-312]{Lev}):
\begin{THEO}
\label{ThHB}
If $\omega \in HB$ then the zeros of $P(z)$ and $Q(z)$ are real and strictly interlace.
\end{THEO}

Similarly, the Hemite-Biehler class $\overline{HB}$ consists of entire functions $\omega$ with no zeros in the 
open lower half-plane $\Im z < 0$ and $|\omega(z)/\bar{\omega}(z)| \leq 1$ for every $z$ with $\Im z >0$. 
An essential difference between $HB$ and  $\overline{HB}$ is that the real and imaginary parts $P$ and $Q$ of 
a function from $\overline{HB}$ may have common real zeros. In general, if $R(z)$ is the canonical product 
corresponding to those common zeros of $\omega \in \overline{HB}$, then $\omega(z)=R(z) \omega_1(z)$ 
with $\omega_1 \in HB$.

Now we are in a position to prove Theorem \ref{MT3}. 

\subsection{Proof of the sufficiency} In this subsection we prove that the condition about the density of the correlation functions 
implies that $\varphi \in \mathcal{LP}$ under milder conditions that do not require 
the Hadmard factorization (\ref{HF}).

\begin{proof}
By Corollary \ref{WCor} the entire function $\varphi$ certainly belongs to $\mathcal{LP}$ provided that, for any fixed $y \in \RR$ ,  
the inequality $W_2 (\phi(\cdot,y);x) + W_2 (\psi(\cdot,y);x) \leq 0$ holds for every $x \in \RR$. 
By Lemma \ref{fipsi}, 
$$
  \phi(x,y) = (\CF \mu_y^\phi)(x) \quad \hbox{and} \quad \psi(x,y) = (\CF \mu_y^\psi)(x).
$$
It follows immediately from Proposition 3.1 in \cite{CsoVar} that when $y$, with $|y|\geq \a$, is fixed, 
 $\phi(x,y)$ and $\psi(x,y)$ are entire functions of the variable $x$ and that they belong to $\mathcal{LP}$. Hence, the 
 Laguerre inequalities (\ref{LI}) hold for them with $j=1$. In other words
\begin{eqnarray*}  
\left[ \phi_x (x,y) \right]^2 - \phi(x,y) \phi_{xx}(x,y) \geq  0, &  \\ 
\left[\psi_x (x,y)\right]^2 - \psi(x,y) \psi_{xx}(x,y) \geq 0, &
\end{eqnarray*}
for every  $z=x+i y \not\in S_\a$.

It remains to establish (\ref{WrLP}) for $z \in S_\a$. 
By Theorem \ref{thm:WronFL}, the Wrosnkians of $\phi(x,y)$ and  $\psi(x,y)$, considered as functions of $x$, 
are the Fourier transforms of the corresponding correlations functions,
$$
W_2 (\phi(\cdot,y);x) = W_2 (\CF \mu_y^\phi;x) = - \CF(\nu_2(\mu_y^\phi) )(x) = -\CF (\nu_2(\cosh(\cdot\, y) K(\cdot)))(x),
$$
and
$$
W_2 (\psi(\cdot,y);x) = W_2 (\CF \mu_y^\psi;x) = - \CF (\nu_2(\mu_y^\psi ) )(x) = - \CF (\nu_2(\sinh(\cdot\, y) K(\cdot)))(x).
$$
By the definition of the correlation function, it is easy to see that 
$$
\nu_2(\cosh(\cdot\, y) K(\cdot); t) + \nu_2(\sinh(\cdot\, y) K(\cdot);t) = \cosh(t y ) \nu_2(K; t). % = \nu_2(\cosh(|\cdot|_1 y ) K(\cdot); t).
$$
Therefore, it follows that 
\begin{align*}
W_2 (\phi(\cdot,y);x) + W_2 (\psi(\cdot,y);x) = - \CF (\cosh(\cdot \, y)\nu_2(K; \cdot) )\, (x).
\end{align*}
We shall prove now that for every  $y \in (-\a,\a)\setminus \{0\}$, 
\begin{equation} \label{FTvu2_Ineqs}
\ \CF (\cosh(\cdot \, y)\nu_2(K;\cdot) )(x) \geq 0 \ \ \mathrm{for\ every}\ x \in  \RR.
\end{equation}
It is clear that the inequalities \eqref{FTvu2_Ineqs} hold for every such $y$ and for $x=0$. Indeed,
\begin{align*}
 [\phi_x (0,y)]^2 - \phi(0,y) \phi_{xx}(0,y) & = \int_\RR \cosh(sy) K(s) ds\int_\RR s^2 \cosh(sy) K(s) ds > 0, \\
 [\psi_x (0,y)]^2 - \psi(0,y) \psi_{xx}(0,y) & = \left( \int_\RR s \sinh(sy) K(s) ds\right)^2  > 0,
\end{align*}
which shows that both $W_2 (\phi(\cdot,y);0)$ and $W_2 (\psi(\cdot,y);0)$ are negative. Since the translates of 
$\cosh(\cdot\, y)\nu_2(K; \cdot)$ are dense in $L^1(\RR)$, then its Fourier transform does not change sign 
for $x\in \RR$, so that the strict inequalities (\ref{FTvu2_Ineqs}) hold for $y \in (-\a,\a)\setminus \{0\}$.  It is clear then 
that (\ref{FTvu2_Ineqs}) also hold for $y=0$ by continuity. 

Finally, observe that, if the $L^1$-density holds for $y=0$ too, then by Wiener's theorem, since $\varphi(x) = \phi(x,0)$, 
$$
[\varphi^\prime(x)]^2 - \varphi(x) \varphi^{\prime\prime}(x) > 0,\ \ x\in \mathbb{R}.
$$   
Therefore, $\varphi$ does not have multiple zeros.
\end{proof}

\subsection{Proof of the necessity}  We shall prove that if $\varphi \in \mathcal{LP}$ is represented as a Fourier transform and 
has an Hadamard factorization (\ref{HF}), then the translations of the 
correlation function are dense in $L^1(\mathbb{R})$ not only for every $y \in (-\a,\a)\setminus \{0\}$ but for every 
$y\in \mathbb{R}\setminus \{0\}$. Because of Wiener's Tauberian Theorem, we need to show that  the {\em strict}
inequalities (\ref{WrLP}) hold. In fact, we shall see that 
$ 
\left[ \phi_x (x,y) \right]^2 - \phi(x,y) \phi_{xx}(x,y) >  0$ and 
$\left[\psi_x (x,y)\right]^2 - \psi(x,y) \psi_{xx}(x,y) > 0$ 
for all $z=x+\i y$, $y\neq 0$, which, in Jensen's terminology means that $\varphi$ is strictly convex with respect to $y$. 
However, we shall employ the Hadamard factorization (\ref{HF}). 

It is worth mentioning that Levin \cite[p. 307]{Lev} pointed out that the requirement $|\omega(z)/\bar{\omega}(z)| < 1$, 
$\Im z >0$ is secured by the fact that the zeros lie in upper half plane when $\omega$ is a polynomial and the same 
holds for entire functions of order zero because of the Phragm\'en-Lindel\"of theorem. In fact, the same holds for functions with 
Hadamard factorization  (\ref{HF}). Indeed, if $\varphi$ is such an entire function in the Laguerre-P\'olya class, and $a >0$,
then both $\varphi(z-\i a)$ and $\varphi(\i a-z)$ belong to $HB$. We sketch the proof that the ``horizontal translation'' 
$\varphi_a(z):=\varphi(z-\i a)\in HB$. Since $\varphi(z) = \varphi(0) \prod (1-z/x_k)$, $x_k \in \mathbb{R}$, then 
$$
\varphi_a(z) = \varphi(0) \varphi(-\i a)  \prod_{k=1}^\infty \left( 1 - \frac{z}{x_k +\i a} \right).
$$
Hence, the zeros of $\varphi_a$ belong to the upper half-plane. Moreover, if $\Im z>0$, $z=x+ \i y$ with $y >0$, 
then 
$$
\left| \frac{\varphi_a(z)}{\overline{\varphi_a(\bar{z})}} \right|
    = \left| \frac{\prod_{k=1}^\infty (1-(x+ \i y)/(x_k +\i a))}{\prod_{k=1}^\infty (1-(x- \i y)/(x_k +\i a))} \right|.
$$
So the quotient of the corresponding terms is $|x_k-x +\i (a-y)|/|x_k-x +\i (a+y)|<1$. This implies immediately that there is 
$q\in (0,1)$, such that the quotients of the finte products above are limited from above by $q$. Thus, 
$| \varphi_a(z)/\overline{\varphi_a(\bar{z})}| <1$ when $\Im z>0$ and $\varphi(z-\i a) \in HB$. Then,
if $\varphi(z-\i a) = P_a(z) + \i Q_a(z)$, the real and imaginary parts $P_a(z)$ and $Q_a(z)$ are (see 
\cite[(3.4)]{CsoVar})
\begin{eqnarray*}
P_a (x)  =  \int_\mathbb{R} K(u) \cosh(au) \cos(xu) du = \phi(x,a),\\  
Q_a (x) =  \int_\mathbb{R} K(u) \sinh(au) \sin(xu) du = \psi(x,a).
\end{eqnarray*}
Since, by Theorem \ref{ThHB}, the zeros of $P_a (x)$ and $Q_a (x)$ are real and strictly interlace, then each has only 
simple zeros and, consequently,  satisfies the strict Laguerre inequalities 
\begin{eqnarray*}  
\left[ \phi_x (x,a) \right]^2 - \phi(x,a) \phi_{xx}(x,a) >  0, &  \\ 
\left[\psi_x (x,a)\right]^2 - \psi(x,a) \psi_{xx}(x,a) > 0. &
\end{eqnarray*}
Although the reasonings up to now concern the case $a>0$, they hold analogously for $a<0$, so that the latter 
inequalities hold for all $a\neq 0$.  The above relations between the Laguerre inequalities and the Wronskians yield that, 
for every fixed $y \neq 0$, 
\begin{equation} \label{FTvu2_Ineqs2}
\ \CF (\cosh(\cdot \, y)\nu_2(K;\cdot) )(x) > 0 \ \ \mathrm{for\ every}\ x \in  \RR.
\end{equation} 
Finally, Theorem \ref{W1} implies that the translates of the the correlation kernel $K_{2,y} (t)$ 
must be dense in $L^1(\mathbb{R})$ for every fixed $y\neq 0$.

It is worth mentioning that without the restriction (iii), one may prove only that  $\varphi(z-\i a)$ and $\varphi(\i a-z)$ belong to $\overline{HB}$, 
as was done in \cite{CsoVar}.  In that case one guarantees that the Laguerre inequalities do hold but not the strict ones and in order to prove the necessity, 
the {\em strict} ones are fundamental.
 
\subsection{Theorem \ref{RH} and further equivalent statements} 

It is clear now that Theorem \ref{RH} follows from Theorem \ref{MT3} by setting $K(t)=\Phi(t)$ and 
having in mind that $\Xi(z)$ obeys all the requirements imposed on $\varphi(z)$. 
  
Wiener's $L^1$ Tauberian theorem has several equivalent formulations, as given in \cite[Theorem 8.1, p. 82]{Kor}, which
can be used to derive other equivalent forms of Theorem \ref{MT3}. We state one such result given in terms of the 
convolution $f*g$.
 
\begin{thm} \label{thm:DT3}
Under the assumption in Theorem {\rm \ref{MT3}}, $\varphi \in \mathcal{LP}$ if and only if,
for every fixed $y\in (-\a,\a)\setminus \{0\}$,  the testing equation $K_{2,y} * g = 0$ for bounded $g$ implies $g = 0$.
\end{thm} 

Recall that $\Phi_{2,y}$ is defined in \eqref{eq:RH}. Applying the above theorem with $K = \Phi$ gives the following corollary. 

\begin{cor}
The Riemann hypothesis is true if and only if, for each $y \in (-1/2,1/2)\setminus \{0\}$, the testing equation $\Phi_{2,y} * g =0$ for bounded $g$ 
implies $g =0$.
\end{cor}

There are two additional equivalent statements of Wiener's theorem in \cite[Theorem 8.1, p. 82]{Kor} that provide further equivalent forms of 
Theorems \ref{MT3} and  \ref{RH}. We omit the details.    
   
\section{Further results concerning Wiener's theorems} 

It is quite clear that Wiener proved his density theorems aiming at a proof of the Prime Number Theorem. 
At the first glance, these theorems are peculiar. They guarantee that the translates of the the Gauss kernels  
$\exp(-a x^2)$, $a>0$, are dense in both $L^1(\mathbb{R})$ and $L^2(\mathbb{R})$. However, according to 
Wiener's results, the translates of the characteristic functions of a compact interval and of the ``tent function'' 
in Example \ref{tent} are dense in $L^2(\mathbb{R})$ but not in $L^1(\mathbb{R})$. Indeed, their Fourier transforms are, 
up to a normalization or scaling, the sinc function $\mathrm{sinc}(t)=\sin t/t$ and its square. The same holds for the 
whole family of the Gegenbauer weights $\omega_\lambda$ in Corollary \ref{oml} whose Fourier transforms are 
$\mathcal{J}_\lambda$ defined in Corollary \ref{cor:Bessel}.  This is so because translations of the argument are permitted,
 but not scaling. However, Theorem \ref{MT3} allows us to build a vast class of functions whose translates 
are dense in  $L^1(\mathbb{R})$. This class is generated by the correlation functions that correspond to 
those $\mathcal{LP}$-functions that obey the requirements in Theorem \ref{MT3}. Let us first state the 
following general consequence of Wiener's theorems and Theorem \ref{thm:WronFL1}:

\begin{thm} \label{thm:density}
Let $f \in L^1(\RR)$ be a nonnegative function, and $t^{2n-2} f \in L^1(\RR)$ for $n \in \NN$, $n \ge 2$. 
\begin{enumerate}
\item $W_n(\CF f; x)$ has no real zeros if, and only if, the span of the translates of the functions 
$\nu_{n,a}(x):= \nu_n(f;x+a)$, $a\in \RR$, are dense in $L^1(\RR)$. 
\item If $\nu_n(f) \in L^2(\RR)$, then the real zeros of $W_n(\CF f; x)$ form a set of measure zero set if, 
and only if, the span of the translates of the functions $\nu_{n,a}(x):= \nu_n(f;x+a)$, $a\in \RR$, are dense 
in $L^2(\RR)$.
\end{enumerate}
\end{thm}

\begin{proof}
By Theorem \ref{thm:WronFL}, $W_n(\CF f; x)$ is equal to, up to a sign, $\CF \nu_n(f)$. By 
Lemma \ref{lem:int-mu}, $\nu_n(f) \in L^1(\RR)$. Hence, (1) follows immediately from Wiener's $L^1$ 
Tauberian theorem (see \cite[Theorem II]{Wien} or Theorem \ref{W1} in the next section). 
Similarly, $(2)$ is a consequence of Wiener's $L^2$ Tauberian theorem \cite[Theorem I]{Wien}. 
\end{proof}

More specifically, Theorem \ref{MT3} immediately yields:

\begin{cor} \label{CorMT3}
Let $\varphi \in \mathcal{LP}$ satisfy the properties {\rm (i)}, {\rm (ii)} and {\rm(iii)} in Section {\rm 3}.  
Then
the translates $\CT (K_{2,y})$ of the kernel 
$$
K_{2,y} (t) = \cosh(ty) \int_{-\infty}^{\infty}  (t-2s)^2 K(t-s) K(s)\, ds
$$
are dense in $L^1(\mathbb{R})$ for every fixed $y\in \mathbb{R} \setminus \{0\}$. Furthermore, if zeros of $\varphi$ are real and simple, then the translates $\CT (K_{2,y})$ of $K_{2,y} (t)$ 
are dense in $L^1(\mathbb{R})$ for every fixed $y\in \mathbb{R}$.
\end{cor}

It is known that both $\mathrm{sinc}^2(t)$ and $\CJ_\l(t)$ (see \cite{DC}) belong to the Laguerre-P\'olya class and
they obey the requirements of Theorem \ref{MT3}; moreover, the zeros of $\CJ_\l(t)$ are simple while  the zeros of 
$\mathrm{sinc}^2(t)$ are obviously double. Then the Corollaries \ref{CorMT3} and \ref{oml} imply the following:  

\begin{cor} The translates of the following functions are dense in $L^1(\RR)$: 
\begin{itemize}
\item[a)]  $\cosh(yt)\,  \nu_2(w_\l;t),\ \ y\in \mathbb{R}$,\  where $ \nu_2(w_\l;t)$ is defined in {\rm (\ref{nuoml})}; 
\item [b)] $\cosh(yt)\, (2-|t|)^3,\ \  y\in \mathbb{R}$;  ($\l=1/2$ of (a))
\item [c)] $\cosh(yt)\, \nu_2(w_\Lambda; t) ,\ \ y\in \mathbb{R}\setminus \{0\}$,\  where $\nu_2(w_\Lambda; t)$ is defined in \rm (\ref{omL}).
\end{itemize}
\end{cor}

Finally, we prove that the translates of the correlation functions $\nu_2(w_{\a,\b};x)$ in Proposition \ref{prop:Jacobi-correlation} are also dense 
in $L^1(\RR)$.

\begin{prop} \label{prop:density}
For $\a,\b > 0$, the family of functions $\{\nu_2(w_{\a,\b};x+a): a \in \RR\}$ is dense in $L^1(\RR)$. Equivalently,
the Wronskian of the Fourier transform of $w_{\a,\b}$, $W_n(\CF w_{a,\b};x)$ does not change sign
on $\RR$. 
\end{prop}

\begin{proof}
As shown in \cite[Theorem 8.1, p. 82]{Kor}, one of the equivalent statements for the density of $\{\nu_2(w_{\a,\b}; x+a): a \in \RR\}$ in 
$L^1(\RR)$ is that the test equation
$$
  \int_{\RR} \nu_2(w_{\a,\b};t) \phi(x-t) dt =0 \qquad \forall x \in \RR,
$$ 
where $\phi$ is a continuous and bounded function, has only trivial solution $\phi (x) =0$. Since $\nu_2(w_{\a,\b})$ is 
nonnegative and has compact support on $[-2,2]$, the measure $d m(t):= \nu_2(w_{\a,\b};t) dt$ is a finite 
nonnegative measure. If $\int_\RR \phi(x-t) d\mu(t) =0$ for all $x \in \RR$ then, by Fatou's lemma, 
$$
  \int_{\RR} \lim_{n\to \infty} \frac{\phi(t+n^{-1}) - \phi(t)}{n^{-1}}  d m(t) \le
       \liminf_{n\to \infty} \int_\RR \frac{\phi(t+n^{-1}) - \phi(t)}{n^{-1}}  d m(t) =0
$$
and the same inequality holds if the left hand side contains a negative sign, from which it follows that 
$\int_\RR \phi'(t)dm(t) =0$ almost everywhere. Hence, $\phi(t) =0$. This proves the density. 
By Theorem \ref{thm:WronFL} and Wiener's $L^1$ Tauberian theorem, the Wronskian of the Fourier transform of 
$w_{\a,\b}$ does not change sign. 
\end{proof}

\end{document}